\documentclass[12 pt]{amsart}

\usepackage{amssymb, amsmath, amsfonts, amsthm, graphicx}
\usepackage[hmargin=1 in, vmargin = 1 in]{geometry} 

\usepackage{mathrsfs} 
\usepackage[capitalise,noabbrev]{cleveref}

\theoremstyle{plain}
\newtheorem{theorem}{Theorem}[section]
\newtheorem{lemma}[theorem]{Lemma}

\theoremstyle{definition}
\newtheorem{definition}[theorem]{Definition}
\newtheorem{example}[theorem]{Example}

\newtheorem{manualtheoreminner}{Conjecture}
\newenvironment{conjecture}[1]{%
  \renewcommand\themanualtheoreminner{#1}%
  \manualtheoreminner
}{\endmanualtheoreminner}

\newcommand{\s}{\sigma}
\newcommand{\q}{\tau}
\renewcommand{\a}{\alpha}
\newcommand{\1}{^{-1}}
\newcommand{\p}{\Phi}
\renewcommand{\b}{\beta}
\renewcommand{\l}{\ell}
\renewcommand{\c}{\cdots}
\newcommand{\jo}{\vee_R}
\newcommand{\w}{\leq_R}
\newcommand{\B}{\mathscr{B}(\Phi^+)}
\newcommand{\Inv}{\mathrm{Inv}}
\newcommand{\inv}{\mathrm{inv}}
\renewcommand{\t}[2]{\left(
        #1\ #2
    \right)}


\title{On a conjecture of Dyer on the join in the weak order of a Coxeter group}
\author{Riccardo Biagioli}
\address{Department of Mathematics, University of Bologna, Bologna, Italy}
\email{riccardo.biagioli2@unibo.it}
\author{Lorenzo Perrone}
\address{Department of Mathematics, University of Bologna, Bologna, Italy}
\email{lorenzo.perrone8@unibo.it}

\begin{document}

\begin{abstract}
In one of his papers on the weak order of Coxeter groups, Dyer formulates several conjectures. Among these, one affirms that the extended weak order forms a lattice, while another offers an algebraic-geometric description of the join of two elements in this poset. The former was recently proven for affine types by Barkley and Speyer. In this paper, we establish the latter for Coxeter groups of types $A$ and $I$. Moreover, we verified the validity of this conjecture for types $H_3$ and $F_4$ through the use of Sage.
\end{abstract}

\maketitle

\section{Introduction}
\emph{Coxeter groups} are abstract groups defined by simple presentations and are of fundamental importance in various areas of mathematics. For instance, dihedral groups and, more generally, the symmetry groups of regular polytopes are classical examples of Coxeter groups. These groups can be realized as reflection groups; in particular, the finite Coxeter groups coincide with the finite Euclidean reflection groups.
Throughout this paper, we assume the reader is familiar with the basic theory of Coxeter groups and we refer to \cite{bjorner2005combinatorics} and \cite{Humphreys} for background material and undefined notation.

Let $(W,S)$ be a Coxeter system, where $S$ is the generating set of the Coxeter group $W$. The group $W$ is defined by relations of the form $(s_is_j)^{m_{ij}}=~e$, for any $s_i,s_j\in S$,
where $m_{ii}=1$, and for $i\neq j$, $m_{ij}=m_{ji}\geq 2$ is either a positive  integer or $\infty$.
This presentation is encoded by the Coxeter graph of $(W, S)$: a labeled graph with vertex set $S$, where an edge labeled $m_{ij}$ connects $s_i$ and $s_j$ whenever $m_{ij} \geq 3$. 
Labels equal to 3 are typically omitted, as they are the most common.

\begin{figure}[h]
    \centering
    \includegraphics[scale=0.9]{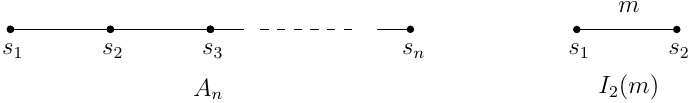} 
    \caption{Coxeter graphs of types $A_n$ and $I_2(m)$.}
\end{figure}

For $w \in W$, the {\em length} of $w$, denoted by $\l(w)$, is the minimum integer $l\in\mathbb{Z}_{\geq0}$ such that $w = s_{i_1}\cdots s_{i_l}$, for some $s_{i_k} \in S$. Any such expression of length $\l(w)$ is called a {\em reduced expression}.

The set of \emph{reflections} of $W$ is given by $T=\{wsw\1\mid w\in W,s\in S\}$ and the elements of $S$ are referred as \emph{simple} reflections.
    
The \emph{Bruhat graph} of $W$, $B(W)$, is the directed graph with vertex set $W$, and where there is an edge from $u$ to $v$, $u\xrightarrow{t} v$, if and only if there is a reflection $t\in T$ such that $v=tu$ and $\l(u)<\l(v)$.
    
One of the most important partial orders on $W$ is the \emph{(right) weak order} which can be defined by the prefix property: $u\w v$ if and only if any reduced expression for $u$ is the prefix of a reduced expression for $v$.
It is well-known that $(W,\w)$ is a meet-semilattice and so a lattice when $W$ is finite. On the contrary, if $W$ is infinite, $(W,\w)$ is never a lattice.

In \cite{dyer_2019}, Dyer introduces a generalization of this poset called extended weak order defined in what follows. Let $\Phi$ be a root system associated with the Coxeter system $(W,S)$, and let $\p^+ \subset \Phi$ denote a set of positive roots. As we will see, each positive root $\a\in\p^+$ corresponds to a reflection $s_\a\in T$.

\begin{figure}[h]
    \centering
    \includegraphics[scale=1.3]{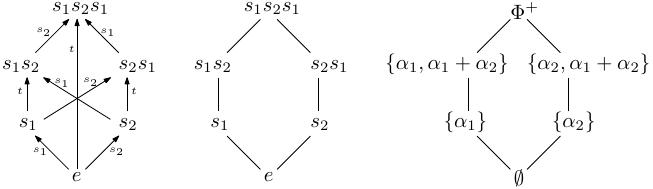}
    \caption{From left to right, the Bruhat graph (where the label $t$ denotes the reflection $s_1s_2s_1$), the Hasse diagrams of $(W,\w)$ and of $(\B,\subseteq)$ for $W$ of type $A_2$.}
    \label{fig:root_biclosed_example}
\end{figure}

\begin{definition}
A subset of positive roots $A\subseteq \p^+$ is \emph{closed} if for any $\a,\b\in A$, 
    \begin{equation*}
        \{a\a+b\b\mid a,b\in\mathbb{R}_{\geq0}\}\cap\p^+\subseteq A,
    \end{equation*}
  is \emph{coclosed} if its complement $A^c$ is closed, and \emph{biclosed} if it is both closed and coclosed.
\end{definition}  
The set of biclosed subsets of $\p^+$, denoted $\B$, ordered by inclusion, forms the so-called \emph{extended weak order} $(\B,\subseteq)$. This poset generalizes the weak order: when $W$ is finite, the two posets are isomorphic, whereas in the infinite case, $(W,\w)$ is isomorphic to a subposet of $(\B,\subseteq)$, for reference see \cite[\S4.1]{dyer_2019}.

In \cite{dyer_2019}, Dyer states two conjectures concerning the extended weak order. The first asserts that this poset is a lattice for every Coxeter system. The second provides a  characterization of the conjectural join of two biclosed sets in $(\B,\subseteq)$.
The former has been recently established for affine types by Barkley and Speyer in \cite{barkley2023affine} thanks to a combinatorial description of biclosed sets they gave in \cite{barkley2022combinatorial}. The latter conjecture, to our knowledge, remains open even for finite Coxeter systems and its original formulation is as follows. 
\smallskip

Let $\tau:\mathscr{P}(\p^+)\rightarrow \mathscr{P}(W)$ be the map from the power set of $\p^+$ to the power set of $W$ sending any $A\subseteq \p^+$ to the set $\q(A)$ defined as
$$\{w \! \in \! W \! \mid \! w \! = \! s_{\a_1}\!\c\! s_{\a_n},\l(s_{\a_1})\!<\!\l(s_{\a_1}s_{\a_2})\!<\!\c\!<\!\l(s_{\a_1}\c s_{\a_n}),\a_1,\ldots,\a_n \! \in \! A \! \}.$$

\begin{conjecture}{D}\cite[\S 2.8]{dyer_2019}, \label{conjecture_1}
    Let $A,B\in\B$; then the join of $A$ and $B$ in $(\B,\subseteq)$ is the following set:
        $$J(A,B)=\{\a\in\p^+\mid s_\a\in\tau(A\cup B)\}.$$
\end{conjecture}

In this paper, we prove \cref{conjecture_1} in the case of Coxeter systems of type $A$ and $I$. More precisely, we work with a reformulation of \cref{conjecture_1} specific for Coxeter groups of finite type, which was communicated to us by Hohlweg~\cite{hohlweg2023problems}. It is worth knowing that using the open-source software Sage \cite{sagemath}, we also verified the validity of Dyer's conjecture for the Coxeter groups of type $H_3$ and $F_4$.

An extended abstract of this paper appears in the proceedings of FPSAC 2025 \cite{Biagioli_Perrone_A}.

\begin{figure}[h]
\centering
    \includegraphics[scale=1]{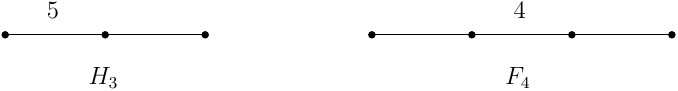}
    \caption{Coxeter graphs of types $H_3$ and $F_4$.}
    \label{fig:}
\end{figure}

\smallskip

\section{Reformulation of Conjecture D}

We begin this section by recalling the concept of \emph{root system} associated with a Coxeter group. Let  $V$ be a real vector space with a basis indexed by the simple reflections of the group, $\Delta=\{\a_s\mid s\in S\}$. Define a symmetric bilinear form $(\ ,\ ):V\times V\rightarrow \mathbb{R}$ by setting
$$(\a_{s_i},\a_{s_j})=-\cos\left(\frac{\pi}{m_{ij}}\right),$$
for all $\a_{s_i},\a_{s_j}\in \Delta$, and extending it by bilinearity to all $V$. The Coxeter group $W$ acts faithfully on $V$ as 
$$s(v) = v-2\frac{(v,\a_s)}{(\a_s,\a_s)}\a_s,$$
for all $s\in S$ and $v \in V$. 
The \emph{root system} associated with $(W,S)$ is defined as the $W$-orbit $\Phi = W(\Delta)=\{w(\a_s)\mid s\in S,w\in W\}$.
The elements of $\Delta$ are called \emph{simple roots} and those of $\Phi^+=\{\sum_{s\in S}c_s\a_s\mid c_s\in\mathbb{R}_{\geq0}\}\cap\Phi$ the \emph{positive roots}. It is well-known that $\Phi=\Phi^+\sqcup\Phi^-$, where $\Phi^-=-\Phi^+$ is the set of \emph{negative roots}.
Now, we introduce an important set of positive roots associated with an element of a Coxeter group.

\begin{definition}
    For any $w\in W$ we define its \emph{inversion set} as: 
    $$\p_w=\p^+\cap w(\p^-)=\{\a\in\p^+\mid \exists\b\in\p^-,\ \a=w(\b)\}.$$
\end{definition}
Inversion sets are natural examples of biclosed subsets of $\p^+$. Indeed, consider $\a,\b\in\p_w$ and suppose there exist $a,b\in\mathbb{R}_{\geq0}$ such that $a\a+b\b\in\p^+$; now, since $w\1(\a),w\1(\b)\in\p^-$, we get 
    $$w\1(a\a+b\b)=aw\1(\a)+bw\1(\b)\in\p^-,$$
    namely, $a\a+b\b\in\p_w$.
    Similarly, if $\a,\b\in\p_w^c$, then $w\1(\a),w\1(\b)\in\p^+$ and, with analogous computations, we get $w\1(a\a+b\b)\in\p^+$, therefore $a\a+b\b\in\p_w^c$.
Actually, Dyer proved that in any Coxeter group the inversion sets are the only finite biclosed sets.  

\begin{lemma}\cite[\S4.1]{dyer_2019} \label[lemma]{lem:biclosed}
    Let $(W,S)$ be any Coxeter system and $A$ a finite subset of $\p^+$. Then $A\in\B$ if and only if there exists $w\in W$ such that $A=\p_w$.
\end{lemma}

\begin{example}
As we can see in \cref{fig:root_biclosed_example}, in the Coxeter group $W$ of type $A_2$, the set $A=\{\a_1,\a_2\}\subseteq\p^+$ is not an element of $\B$ since $\a_1+\a_2\notin A$, i.e. $A$ is not closed. This also implies that $B=\{\a_1+\a_2\}\subseteq\p^+$ is not coclosed, therefore $B\notin\B$. Indeed, both $A$ and $B$ are not inversion sets of any element of $W$.   
\end{example}

Inversion sets have several important properties. For any $w\in W$, we have $\l(w)=|\p_w|$, moreover, they characterize the weak order: for any $u,v\in W$, 
\begin{equation}\label{eq:char_weak_order}
u\w v \mbox{ if and only if }\p_u\subseteq\p_v.
\end{equation}
As mentioned earlier, when $W$ is finite, the posets $(W,\w)$ and $(\B,\subseteq)$ are isomorphic. More precisely, by \cref{lem:biclosed}, this isomorphism is realized by the map $w\mapsto \p_w$. 
This map remains an injective poset morphism even when $W$ is infinite. An example illustrating this correspondence is shown in \cref{fig:root_biclosed_example}. 
It is worth noting that when $W$ is infinite, the extended weak order continues to admit a maximal element which is $\p^+\in\B$, whereas the weak order does not.
\smallskip

To reformulate the conjecture, we use the classical bijection $\varphi:\p^+\rightarrow T$ (see \cite[\S1.14]{Humphreys}) which assigns to any posivite root $\a\in \p^+$ the reflection 
$$\varphi(\a)=wsw\1=s_\a,$$
where $\a = w(\a_s)$, with $\a_s$ a simple root. 
Under this bijection, the inversion set $\p_w$ is mapped to the \emph{left-reflection set} of $w$ defined by 
$$T_L(w)=\{t\in T\mid \l(tw)<\l(w)\}.$$ 
Therefore, from \eqref{eq:char_weak_order} we have 
\begin{equation}\label{eq:inversion-inclusion}
    u\w v\mbox{ if and only if }T_L(u)\subseteq T_L(v).
\end{equation}
In particular, every element $w\in W$ is uniquely determined by $T_L(w)$ or equivalently by $\p_w$.

We now introduce a concept that plays a fundamental role throughout this paper.

\begin{definition}[$(u,v)$-Bruhat path]
    Let $u,v\in W$. A $(u,v)$-\emph{Bruhat path} is any (directed) path in Bruhat graph $B(W)$ starting from the vertex $e$ and whose edges have labels in the set $T_L(u)\cup T_L(v)$. 
    We denote by $V_W(u,v)$ the set of vertices of all the $(u,v)$-Bruhat paths in $B(W)$.
\end{definition}

We are now ready to state the reformulation of the conjecture. Recall that in a poset $(P,\leq)$, the {\em join} of two elements $x,y\in P$, if it exists, is the unique element $x\vee y\in P$, satisfying the following properties:
\begin{itemize}
    \item $x,y\leq x\vee y$;
    \item for any $z\in P$ such that $x,y\leq z$, then $x\vee y\leq z$.
\end{itemize}

\begin{conjecture}{H}\label{conjecture_2}
    Let $W$ be a finite Coxeter group and $u,v\in W$. Then
        $$T_L(u\jo v)=T\cap V_W(u,v).$$
\end{conjecture}

\cref{conjecture_2} states that the left-reflection set of the join $u\jo v$ is precisely the set of reflections reached by all possible $(u,v)$-Bruhat paths. 

The next result states the equivalence between \cref{conjecture_1} and \cref{conjecture_2}. Before presenting a proof, we illustrate the relevant notions introduced above through the following example. As usual, from now on, considering $S=\{s_1,\ldots,s_n\}$, we adopt the convention $\a_i=\a_{s_i}$.

\begin{example}
    Let $\p^+=\{\a_1,\a_2,\a_3,\a_1+\a_2,\a_2+\a_3,\a_1+\a_2+\a_3\}$ be the set of positive roots of the Coxeter group of type $A_3$. Let us consider the following two biclosed sets $A=\{\a_1,\a_1+\a_2\},\ B=\{\a_1,\a_3\}\in\B$. It is easy to see that the join of $A$ and $B$ in the poset $(\B,\subseteq)$ is the set 
    $A\vee B= \{\a_1,\a_3,\a_1+\a_2,\a_1+\a_2+\a_3\}$. We now verify that this set coincides with $J(A,B)$, as defined in the statement of \cref{conjecture_1}. We notice that $A\cup B=\{\a_1,\a_3,\a_1+\a_2\}$, so $s_{\a_1},s_{\a_3},s_{\a_1+\a_2}\in \q(A\cup B)$, which gives $$A\cup B=\{\a_1,\a_3,\a_1+\a_2\}\subseteq J(A,B).$$
    Note that, $s_{\a_1}=s_1$, $s_{\a_3}=s_3$, so we get $s_{\a_1+\a_2}=s_1s_2s_1$, as $\a_1+\a_2=s_1(\a_2)$. 
    Moreover, by the same reasoning, we can write $$s_{\a_1+\a_2+\a_3}=s_{\a_3}s_{\a_1+\a_2}s_{\a_3}=s_3(s_1s_2s_1)s_3;$$
    where $\l(s_{\a_3})<\l(s_{\a_3}s_{\a_1+\a_2})<\l(s_{\a_3}s_{\a_1+\a_2}s_{\a_3})$, thus we get $\a_1+\a_2+\a_3\in J(A,B)$. 
    Finally, one can check that $s_{\a_2}$ and $s_{\a_2+\a_3}$ cannot be written as an increasing-length product of $s_{\a_1},s_{\a_3},s_{\a_1+\a_2}$ and obtain that $J(A,B)=A\vee B$, as stated in \cref{conjecture_1}.

    Now, we reinterpret the same example in the setting of \cref{conjecture_2}. Consider $u=s_1s_2$ and $v=s_1s_3$ in $V$ whose inversion sets correspond to the biclosed sets $A$ and $B$, respectively. That is, $\p^+_u= A$ and $\p^+_v=B$. 
By definition, all the $(u,v)$-Bruhat paths have labels in $T_L(u)\cup T_L(v)=\{s_1,s_3,s_1s_2s_1\}$ and are depicted in \cref{fig:formulation_H}. The set of reflections reached by these paths is $\{s_1,s_3,s_1s_2s_1,s_1s_2s_3s_2s_1\}$ which is the left-reflection set of the element $s_1s_2s_3s_2$, the join $u\jo v$ in $(W,\w)$, as predicted by \cref{conjecture_2}. As a final remark, observe that $\p_{u\jo v}$ is exactly the set $A\vee B$ computed previously, confirming the equivalence of the two conjectures in this example.

    \begin{figure}[h]
        \centering
        \includegraphics[scale=1.3]{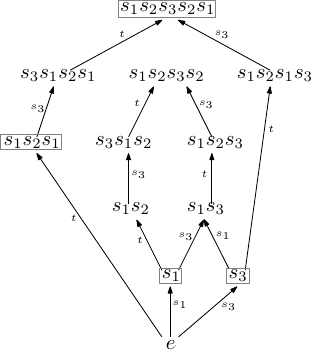}
        \caption{All the possible $(s_1s_2,s_1s_3)$-Bruhat paths in a Coxeter group of type $A$, where the label $t$ denotes the reflection $s_1s_2s_1$. We highlighted the elements of $T\cap V_W(s_1s_2,s_1s_3)$.}
        \label{fig:formulation_H}
    \end{figure}
\end{example}

We are now ready to show that the two formulations of the conjectures are equivalent in the finite case.

\begin{theorem}
    Let $W$ be a finite Coxeter group. Then \cref{conjecture_1} and \cref{conjecture_2} are equivalent.
\end{theorem}

\begin{proof}
    Since $W$ is a finite Coxeter group, then by \cref{lem:biclosed} we know that any element of $\B$ is of the form $\p_w$ for some $w\in W$. So, if we consider $\p_u,\p_v\in\B$; then \cref{conjecture_1} says that 
    \begin{equation}\label{eq:formulation_1}
        \p_u\vee\p_v=\{\a\in\p^+\mid s_\a\in\tau(\p_u\cup\p_v)\}=\{\a_t\in\p^+\mid t\in\tau(\p_u\cup\p_v)\}.
    \end{equation}
    The posets $(W,\w\nobreak)$ and $(\B,\subseteq)$ are isomorphic, so $\p_u\vee\p_v=\p_{u\jo v}$ and this set corresponds to the left-reflection set $T_L(\s\jo\q)$ through the bijection $\varphi$.
    
    We now unpack equation~\eqref{eq:formulation_1} using the definition of the function $\tau$. A reflection $t$ is in $\tau(\p_u\cup\p_v)$ if there exist $\a_1,\ldots,\a_n \in \p_u\cup\p_v$ such that $t=s_{\a_1}\c s_{\a_n}$ and 
    \begin{equation}\label{eq:formulation_2}
        \l(s_{\a_1})<\l(s_{\a_1}s_{\a_2})<\c<\l(s_{\a_1}\c s_{\a_n}).
    \end{equation}
    Hence, it is sufficient to show that this condition on a reflection $t\in T$ is equivalent to $t\in V_W(u,v)$.
    Since $t$ is a reflection, $t=t\1=s_{\a_n}s_{\a_{n-1}}\c s_{\a_1}$
    and equation~\eqref{eq:formulation_2} implies $\l(s_{\a_1})<\l(s_{\a_2}s_{\a_1})<\c<\l(s_{\a_n}\c s_{\a_1}).$
    Finally, recalling the correspondence between $\p_w$ and $T_L(w)$, we get $$\a_1,\ldots,\a_n \in \p_u\cup\p_v\iff s_{\a_1},\ldots, s_{\a_n}\in T_L(u)\cup T_L(v),$$ 
    therefore, $t\in V_W(u,v)$ as we have the path
    $$e\xrightarrow{s_{\a_1}}s_{\a_1}\xrightarrow{s_{\a_2}}s_{\a_2}s_{\a_1}\xrightarrow{s_{\a_3}}\c\c\xrightarrow{s_{\a_n}}s_{\a_n}\c s_{\a_1}=t.$$
    So, \cref{conjecture_1} and \cref{conjecture_2} are equivalent whenever $W$ is a finite Coxeter group.
\end{proof}

\smallskip
\section{Dihedral group}

In this section, we prove \cref{conjecture_2} in the case of the finite dihedral group, $I_2(m)$. This group provides the simplest nontrivial example for the conjecture, as it is generated by only two simple reflections $S=\{s,r\}$ with defining relations $(sr)^m=e=s^2=r^2$. 

We begin by analyzing two easy instances in $I_2(4)$, the symmetry group of the square. Here we have $s\jo r=srsr$ and $T_L(s)\cup T_L(r)=\{s,r\}$; therefore any reflection contained in $T=\{s,r,srs,rsr\}$ is a vertex of a $(s,r)$-Bruhat path. Indeed, we have the following two paths:
$$e\xrightarrow{s}s\xrightarrow{r}rs\xrightarrow{s}srs,\qquad e\xrightarrow{r}r\xrightarrow{s}sr\xrightarrow{r}rsr.$$
Hence, $T\cap V_{I_2(4)}(s,r)=T=T_L(srsr)$.

Now, consider $s\jo srs=srs$ and observe that $T_L(s)\subseteq T_L(srs)=\{s,srs,rsr\}$, so any $(s,srs)$-Bruhat path has labels in $T_L(srs)$. This implies
    $$\{s,srs,rsr\}\subseteq T\cap V_{I_2(4)}(s,srs).$$
    Furthermore, since a simple reflection is in a $(u,v)$-Bruhat path if and only if it is an element of $T_L(u)\cup T_L(v)$, we get $r\notin V_{I_2(4)}(s,srs)$; thus, we obtain $T_L(s\jo srs)=T\cap V_{I_2(4)}(s,srs)$.

\begin{theorem}
    Let $u,v\in I_2(m)$; then
    $$T_L(u\jo v)=T\cap V_{I_2(m)}(u,v).$$
\end{theorem}

\begin{proof}
    First, suppose that $u\nleq_R v$ and $v\nleq_R u$. In this case, their join is the maximal element $w_0$, so $u\jo v=w_0$. Since $T_L(u\jo v)=T$,
    we must verify that every reflection in $T$ appears as a vertex of a $(u,v)$-Bruhat path. This follows easily from the observation that the reduced expressions of $u$ and $v$ begin with two different generators. Thus, both $r$ and $s$ lie in $T_L(u)\cup T_L(v)$. As $s$ and $r$ generate the whole group $I_2(m)$, any element in $T$ can be reached by a $(u,v)$-Bruhat path, so $T=T(u\jo v)=T\cap V_{I_2(m)}(u,v)$.
    
    We are left to prove the other case: $u\w v$ or $v\w u$. We can suppose that both $u\neq w_0$ and $v\neq w_0$, otherwise we would get $\{s,r\}\subseteq T_L(u)\cup T_L(v)$ and we could conclude as above.
    Assume, without loss of generality, that $u\w v$ and the reduced expression of $u$ starts with $s$.
    Then the reduced expression of $v$ is either $(sr)^h$ or $(sr)^hs$ for some $h\in\mathbb{Z}_{\geq0}$ and, in both cases,
    \begin{equation}\label{eq:TL_v}
        T_L(v)=\{s,srs,\ldots,(sr)^ds\},
    \end{equation}
    for some $d\in\mathbb{Z}_{\geq0}$, where some expressions may not be reduced.
    Since $u\jo v=v$, we have only to check that $T_L(v)=T\cap V_{I_2(m)}(u,v).$
    Furthermore, $u\w v$ is equivalent to $T_L(u)\subseteq T_L(v)$, so, in this case, $(u,v)$-Bruhat paths have labels in $T_L(v)$. Naturally, $T_L(v)\subseteq V_{I_2(m)}(u,v),$ so it suffices to check that $(T\setminus T_L(v))\cap V_{I_2(m)}(u,v)=\emptyset.$
    We split the proof in two cases. 
    
    \textbf{Case 1:} all the expressions in the set in \eqref{eq:TL_v} are reduced, equivalently, $d\leq\frac{m-1}{2}$.\\
    We need to check which vertices we can get in a generic $(u,v)$-Bruhat path. We can take a look at the Bruhat graphs in \cref{fig:Bruhat_Im}.
    
    \begin{figure}[h]
        \includegraphics[scale=1]{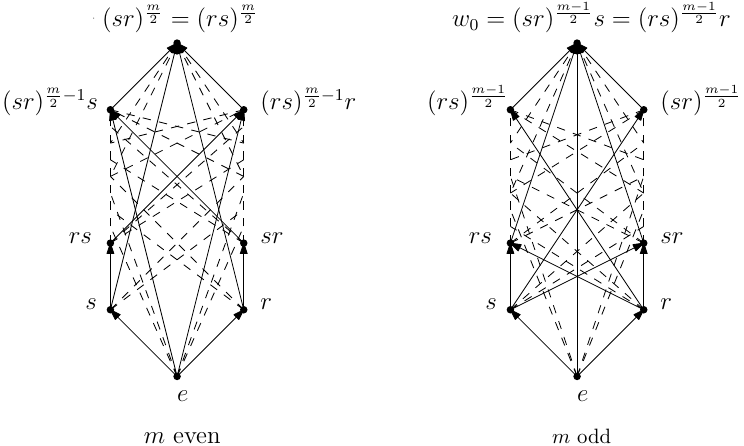}
        \caption{The Bruhat graph $B(I_2(m))$ for $m$ even and odd: the edge labels are omitted.}
        \label{fig:Bruhat_Im}
\end{figure}
    Let $(sr)^is$ and $(sr)^js$ be respectively the first and the second labels of a Bruhat path. Then the third vertex of the path is
    \begin{equation}\label{eq:1step}
        (sr)^js(sr)^is=(sr)^j(rs)^i=
        \begin{cases}
            (sr)^{j-i}, \quad&\emph{ if }j>i,\\
            (rs)^{i-j}, &\emph{ otherwise},
        \end{cases}
    \end{equation}
    where the last equality shows the two possible reduced expressions. Moreover, by definition of Bruhat path, we also require that the length increases from a vertex to its subsequent, hence, in equation~\eqref{eq:1step} we ask that
    $$\l((sr)^is)<\l((sr)^js(sr)^is)\iff
        \begin{cases}
            2i+1<2(j-i), \quad&\emph{ if }j>i,\\
            2i+1<2(i-j), &\emph{ otherwise}.
        \end{cases}$$
    The second inequality is necessarily false, so, the reduced expression of the third vertex can only be of the form $(sr)^l$, with $l<d$. Let $(sr)^ks$ be the third label of the path; then the next vertex is
    $$(sr)^ks(sr)^l=(sr)^k(rs)^{l-1}r=
        \begin{cases}
            (sr)^{k-l}s, \quad&\emph{ if }k>l-1,\\
            (rs)^{l-k-1}r, &\emph{ otherwise},
        \end{cases}$$
    where, of course, we have to require that the length increases, i.e.
    $$\l((sr)^l)<\l((sr)^ks(sr)^l)\iff
        \begin{cases}
            2l<2(k-l)+1, \quad&\emph{ if }k>l-1,\\
            2l<2(l-1-k)+1, &\emph{ otherwise}.
        \end{cases}$$
    Again, the second inequality is false, so we can only get a vertex whose reduced expression is $(sr)^bs$, with $b< d$. 
    At this point, it is easy to note that a repetition of the same argument shows that any reflection reached by a $(u,v)$-Bruhat path has reduced expression of the form $(sr)^bs$ with $b\leq d$. Therefore, since the reflections in $T\setminus T_L(v)$ have reduced expressions $(rs)^ir$ or $(sr)^js$, with $d<j\leq\frac{m-1}{2}$, they cannot be vertices of a $(u,v)$-Bruhat path. Hence, $T\cap V_{I_2(m)}(u,v)=T_L(v)$.

    \textbf{Case 2:} the set in \eqref{eq:TL_v} contains expressions that are not reduced.\\ Now, if $m$ is even, we rewrite the elements of $T_L(v)$ only using reduced expressions as
        $$T_L(v)=\{s,srs,\ldots,(sr)^{\frac{m}{2}-1}s,(rs)^{\frac{m}{2}-1}r,(rs)^{\frac{m}{2}-2}r\ldots,(rs)^{m-d-1}r\},$$
    whereas if $m$ is odd we write it as
        $$T_L(v)=\{s,srs,\ldots,(sr)^{\frac{m-1}{2}}s=(rs)^{\frac{m-1}{2}}r,(rs)^{\frac{m-3}{2}}r\ldots,(rs)^{m-d-2}r\}.$$
    We discuss both cases at the same time. Note that any reflection whose reduced expression starts with $s$ is an element of $T_L(v)$. We consider 
    $$k:=\min\{n\in\mathbb{Z}_{\geq0}\mid (rs)^nr\in T_L(v)\}$$
    and prove that if $0\leq j<k$, then $(rs)^jr\notin V_{I_2(m)}(u,v)$.\\
    Suppose $(rs)^jr\in V_{I_2(m)}(u,v)$, with $0\leq j<k$; then the first label of the path that reaches this reflection must be $(sr)^is$ for some $i<j$, since we need to start with a reflection that is shorter in length. At this point, we have two possibilities for the second label of the path. If it is $(rs)^lr$, for some $k\leq l\leq \frac{m-1}{2}$, then the third vertex of the path is
    \begin{equation}\label{eq:cases_dihedral}
        (rs)^lr(sr)^is=(rs)^l(rs)^{i+1}=
        \begin{cases}
            (rs)^{l+i+1}, \quad&\emph{ if }l+i+1\leq\frac{m}{2},\\
            (sr)^{m-l-i-1}, &\emph{ otherwise}.
        \end{cases}
    \end{equation}
    Note that, since $l\geq k$, the first case of equation~\eqref{eq:cases_dihedral} gives an element which is already longer than $(rs)^jr$, so we do not consider it. \\
    If the second label is $(sr)^ls$, for some $i<l\leq \frac{m-1}{2}$, then, as we computed earlier, the reduced expression of the third vertex is $(sr)^b$.
    So, whatever is the reduced expression of the second label, in order for $(rs)^jr$ to be a vertex of the path, the third vertex is $(sr)^b$ for some $b$. Starting from this vertex, if we label the third edge by $(sr)^ls$, with $l\leq\frac{m-1}{2}$, then, as in case 1, we get a fourth vertex whose reduced expression is $(sr)^cs$, for some $c$. Hence, we would be back to a vertex with the same form of the second one. For this reason, we only have to discuss the other possibility. If the third label is $(rs)^lr$, with $k\leq l\leq \frac{m-1}{2}$, then the fourth vertex of the path is
    \begin{equation}\label{eq:cases_dihedral2}
        (rs)^lr(sr)^b=(rs)^{l+b}r=
        \begin{cases}
            (rs)^{l+b}r, \quad &\emph{ if }l+b\leq\frac{m-1}{2},\\
            (sr)^{m-l-b-1}s, \quad&\emph{ otherwise}.
        \end{cases}
    \end{equation}
    The first case of equation~\eqref{eq:cases_dihedral2} is a reduced expression of the form we are looking for, but its length is necessarily more than $2k+1$. This means that we cannot get $(rs)^jr$, with $j<k$ as a vertex of a $(u,v)$-Bruhat path.
    This concludes the proof, as $T_L(v)=T\cap V_{I_2(m)}(u,v)$ and $T_L(u\jo v)=T_L(v)$.
\end{proof}

\section{Symmetric group}
In this section, we consider the usual combinatorial description of the Coxeter system $(W,S)$ of type $A_{n-1}$. In this framework, $W$ is the symmetric group $S_n$ of permutations of the set $[n]=\{1,2,\ldots,n\}$, $S=\{({i}\ {i+1})\mid i\in[n-1]\}$ is the set of simple transpositions and $T=\{\t{i}{j}\mid 1\leq i<j\leq n\}$ is the set of reflections. We represent a permutation with its \emph{one-line} notation $\s=\s(1)\s(2)\c\s(n)$ and we denote by 
$$\Inv(\s)=\{ (i,j)\in [n]\times[n] \mid i<j \ {\rm and} \ \s(i)>\s(j)\}$$ 
the set of \emph{inversion} of $\s$, whose cardinality is denoted by $\inv(\s)$ and coincides with $\l(\s)$, the Coxeter length of $\s$.
The following result is well-known.
\begin{lemma}\label[lemma]{lem:inversions}
  Let $i,j\in[n]$ be such that $i<j$ and $\s\in S_n$. Then the following are equivalent:  
\begin{itemize}
    \item[(i)] $\t{a}{b}\in T_L(\s)$;
    \item[(ii)] $(\s\1(b),\s\1(a))\in \Inv(\s)$;
    \item[(iii)] $(a,b)\in \Inv(\s\1)$.
\end{itemize}
\end{lemma}

Therefore, for $\s,\q\in S_n$, equation \eqref{eq:inversion-inclusion} becomes $$\s\w\q \ \mbox{if and only if} \ \Inv(\s\1)\subseteq \Inv(\q\1).$$
Before proving the conjecture for symmetric groups, we check its statement in the following example.

\begin{figure}[h]
    \centering
    \includegraphics[scale=1]{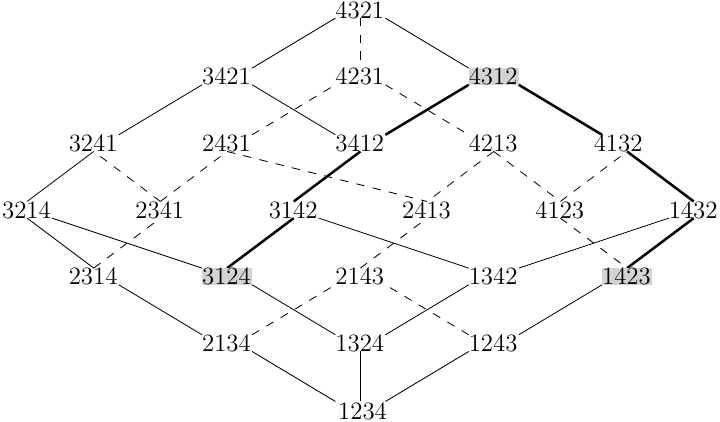} 
    \caption{Hasse diagram of $(S_4,\w)$ from example \ref{ex:example_S4}.}
    \label{fig:Weak_S_4}
\end{figure}

\begin{example}\label{ex:example_S4}
    Consider $\s=3124$ and $\q=1423$ in $S_4$. We have $T_L(\s)=\{\t{1}{3},\t{2}{3}\}$, $T_L(\q)=\{\t{2}{4},\t{3}{4}\}$ and from \cref{fig:Weak_S_4} we see that $\s\jo\q=4312$, so we compute $T_L(\s\jo\q)=\{\t{1}{3},\t{1}{4},\t{2}{3},\t{2}{4},\t{3}{4}\}.$
    Let us verify that the conjecture holds for this example, i.e. $T_L(\s\jo\q)=T\cap V_{S_4}(\s,\q).$
    
    By definition of $(\s,\q)$-Bruhat path, we have $T_L(\s)\cup T_L(\q)\subseteq T\cap V_{S_4}(\s,\q),$
    as $T_L(\s)\cup T_L(\q)$ is the set of labels of $(\s,\q)$-Bruhat paths; furthermore we have the path
    $$1234\xrightarrow{\t{1}{3}}3214\xrightarrow{\t{3}{4}}4213\xrightarrow{\t{1}{3}}4231=\t{1}{4},$$
    thus, $\t{1}{4}\in V_{S_4}(\s,\q)$ and $T_L(\s\jo\q)\subseteq T\cap V_{S_4}(\s,\q).$
    Now, since $T\setminus T_L(\s\jo\q)=\{\t{1}{2}\}$, we can look at all possible $(\s,\q)$-Bruhat paths which are represented in \cref{fig:Bruhat_paths} and check that $\t{1}{2}\notin V_{S_4}(\s,\q)$, so the conjecture holds. We can also argue that $\t{1}{2}$ is a simple reflection, therefore it is a vertex of a $(\s,\q)$-Bruhat path if and only if it is in $T_L(\s)\cup T_L(\q)$.
\end{example}

\begin{figure}[h]
    \centering
     \centering
        \includegraphics[scale=1]{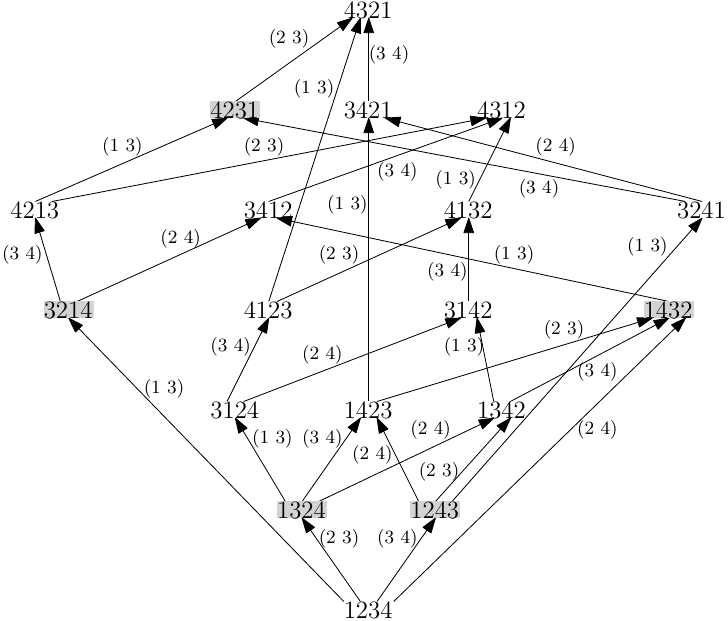}
        \caption{A subgraph of $B(S_4)$ representing every possible $(\s,\q)$-Bruhat path where reflections are highlighted.}
        \label{fig:Bruhat_paths}
\end{figure}

In order to prove \cref{conjecture_2} we use a known result about left-reflection sets of the join of two permutations. First, we say a set of transpositions $J\subseteq T$ is \emph{transitively closed} if for any $\t{i}{j},\t{j}{k}\in J$ with $i<j<k$, then $\t{i}{k}\in J$.
By \cref{lem:inversions},   $\t{i}{j}\in T_L(\s)$ if and only if $i<j$ and $\s\1(i)>\s\1(j)$, so $T_L(\s)$ is transitively closed.
The following theorem is well-known and its proof can be found in \cite[Theorem 1(b)]{markowsky1994permutation}.
\begin{theorem}\label{thm:transitive_T_L}
Let $\s,\q\in S_n$ and denote by $J^{tc}$ the transitive closure of a subset $J\subseteq T$; then
    $$T_L(\s\jo\q)=(T_L(\s)\cup T_L(\q))^{tc}.$$
\end{theorem}

We prove the conjecture by double inclusion. We start by showing that any left-reflection of the join $\s\jo\q$ is a vertex of a $(\s,\q)$-Bruhat path. A Bruhat path is called \emph{palindromic} if its sequence of labels is palindromic.

\begin{theorem}\label[lemma]{lem:first_inclusion_S_n}
    Let $\s,\q\in S_n$, then 
    $$T_L(\s\jo\q)\subseteq T\cap V_{S_n}(\s,\q).$$
    In particular, for any $t\in T_L(\s\jo\q)$ there is a palindromic $(\s,\q)$-Bruhat path from $e$ to $t$. 
\end{theorem}
\begin{proof}
    Let $t=\t{a}{b}$, with $1\leq a<b\leq n$ and suppose $t\in T_L(\s)\cup T_L(\q)$; then $e\xrightarrow{t}t$ is a palindromic $(\s,\q)$-Bruhat path.
    On the other hand, if $t\notin T_L(\s)\cup T_L(\q)$, by Theorem~\ref{thm:transitive_T_L}, there exists a chain $a=i_0<i_1<i_2<\c<i_{l-1}<i_l=b,$ with $l\geq 2,$
    such that $\t{i_r}{i_{r+1}}\in T_L(\s)\cup T_L(\q),$ for any $r\in\{0,1,\ldots,l-1\}$. 
    Furthermore, $t$ is obtained by product of transpositions as follows:
    $$t=\t{a}{b}=\t{a}{i_1}\t{i_1}{i_2}\c\t{i_{l-2}}{i_{l-1}}\t{i_{l-1}}{b}\t{i_{l-2}}{i_{l-1}}\c\t{a}{i_1}.$$
    Now, in order to prove the theorem is sufficient to show that the palindromic path 
    \begin{equation*}
            e\xrightarrow{\t{a}{i_1}}\t{a}{i_1}\xrightarrow{\t{i_1}{i_2}}\t{i_1}{i_2}\t{a}{i_1}\xrightarrow{\t{i_2}{i_3}}
            \c\xrightarrow{\t{i_1}{i_2}}\t{a}{i_1}\t{a}{b}\xrightarrow{\t{a}{i_1}}\t{a}{b}.
    \end{equation*}
    is a $(\s,\q)$-Bruhat path. We do this in two steps.
  \smallskip
  
    \textbf{Step 1:} Let us consider for $k\in[l-1]$, 
    \begin{equation}\label{eq:new_notation_1}
        \s_k=\t{i_k}{i_{k+1}}\c\t{i_1}{i_2}\t{a}{i_1}= \begin{bmatrix}
        a & i_1 & i_2 & \c & i_k & i_{k+1}\\
        i_{k+1} & a & i_1 & \c & i_{k-1} & i_k
        \end{bmatrix},
    \end{equation}
    where on the right-hand side we display only the indices appearing in the two-line notation of $\s_k$ that are relevant in our computation. In particular, all the indices that do not appear are fixed points.
    Similarly, we have
    \begin{equation}\label{eq:new_notation_2}
        \s_{k-1}=\t{i_{k-1}}{i_k}\c\t{i_1}{i_2}\t{a}{i_1}= \begin{bmatrix}
        a & i_1 & i_2 & \c & i_k & i_{k+1}\\
        i_k & a & i_1 & \c & i_{k-1} & i_{k+1}
        \end{bmatrix}.
    \end{equation}
    We need to check that for any $k\in[l-1]$, 
    \begin{equation}\label{eq:length_growing}
        \inv(\s_k)>\inv(\s_{k-1}).
    \end{equation}
    From \eqref{eq:new_notation_1} and \eqref{eq:new_notation_2}, we see that 
    \begin{equation}\label{eq:inv_count_1}
        \begin{split}
            &\inv(\s_k)=\inv(\s_{k-1})+(i_{k+1}-i_k)+(i_{k+1}-i_k-1)
        \end{split}
    \end{equation}
    as in $\s_k$ we have $(i_{k+1}-i_k)$ more inversions in which the first index is $a$ and $(i_{k+1}-i_k-1)$ more inversions in which the second index is $i_{k+1}$ and the first is not $a$. Thus, equation~\eqref{eq:inv_count_1} yields $\inv(\s_k)=\inv(\s_{k-1})+2(i_{k+1}-i_k)-1$
    and, since $2(i_{k+1}-i_k)-1>0$, for any $k\in[l-1]$, the inequality in \eqref{eq:length_growing} is satisfied.
\smallskip

    \textbf{Step 2:} we call $\q_k:=\t{i_{l-k}}{i_{l-k+1}}\c\t{i_{l-2}}{i_{l-1}}\t{i_{l-1}}{b}\c\t{i_1}{i_2}\t{a}{i_1}$ and check that $\inv(\q_k)>\inv(\q_{k-1})$, for any $k\in\{2,3,\ldots,l\}$.
    We use the notation introduced earlier getting
    $$\q_{k-1}=\begin{bmatrix}
        a & i_1 & i_2 & \c & i_{l-k+1} & b\\
        b & a & i_1 & \c & i_{l-k} & i_{l-k+1}
        \end{bmatrix}$$
    and
    $$\q_k=\begin{bmatrix}
        a & i_1 & i_2 & \c & i_{l-k} & i_{l-k+1} & b\\
        b & a & i_1 & \c & i_{l-k-1} & i_{l-k+1} & i_{l-k}
        \end{bmatrix}.$$
    We note that $\q_k$ has $(i_{l-k+1}-i_{l-k}-1)$ less inversions in which the last index is $i_{l-k+1}$ and $(i_{l-k+1}-i_{l-k})$ more inversions in which the second index is $b$. Thus, 
    $$\inv(\q_k)=\inv(\q_{k-1})-(i_{l-k+1}-i_{l-k}-1)+(i_{l-k+1}-i_{l-k})=\inv(\q_{k-1})+1,$$
    therefore, $\inv(\q_k)>\inv(\q_{k-1})$, for any $k\in[l]$.
    This concludes the proof.
\end{proof}

It remains to prove the inclusion $T\cap V_{S_n}(\s,\q)\subseteq T_L(\s\jo\q)$.
In order to do that, we first establish a property of Bruhat paths starting from $e$ that reach a reflection $\t{a}{b}\in T$. Although the following result is probably known in some form, we are unaware of a suitable reference, so we include a proof here. 

\begin{lemma}\label[lemma]{lem:T_ab}
    Let $\t{a}{b}\in S_n$ with $1 \leq a<b\leq n$; then all the edges of a Bruhat path from $e$ to $\t{a}{b}$ are labeled by elements of $T_{ab}:=\{\t{i}{j}\mid a\leq i<j\leq b\}.$
\end{lemma}
\begin{proof}
    Let $\t{i_1}{j_1},\t{i_2}{j_2},\ldots,\t{i_h}{j_h}\in T$ be the labels of a Bruhat path from $e$ to $\t{a}{b}$ listed from last to first, i.e. $\t{a}{b}=\t{i_1}{j_1}\c\t{i_{h-1}}{j_{h-1}}\t{i_h}{j_h}.$\\
    We prove that $\t{i_r}{j_r}\in T_{ab}$, for any $r\in[h]$, by induction on $r$. 
    Since $\t{a}{b}$ is the last vertex of the path, then the label of the last edge $\t{i_1}{j_1}$ is an element of $T_L(\t{a}{b})$.

It is easy to note that $T_L(\t{a}{b})=\{\t{a}{j}\mid a<j\leq b\}\cup\{ \t{i}{b}\mid a< i<b\};$
    therefore, $T_L(\t{a}{b})\subseteq T_{ab}$, thus, $\t{i_1}{j_1}\in T_{ab}$ verifying the base case of induction.
    Suppose $r>1$ and assume, by inductive hypothesis, that $\t{i_l}{j_l}\in T_{ab}$, for any $l\in[r-1]$. Since the length in a Bruhat path must increase at every step, we know that 
    \begin{equation}\label{eq:lemma_Bruhat_path}
        \t{i_r}{j_r}\in T_L(\t{i_{r-1}}{j_{r-1}}\c\t{i_1}{j_1}\t{a}{b}).
    \end{equation}
    Call $\b:=\t{i_{r-1}}{j_{r-1}}\c\t{i_1}{j_1}\t{a}{b}$ and suppose $\t{i_r}{j_r}\notin T_{ab}$ and note that, since we assume $i_r<j_r$, by \cref{lem:inversions}, \eqref{eq:lemma_Bruhat_path} is equivalent to 
    \begin{equation}\label{eq:lemma_Bruhat_path_2}
        \b\1(i_r)>\b\1(j_r).
    \end{equation} 
    Recall that by inductive hypothesis, $i_k,j_k\in\{a,a+1,\ldots,b\},$ for any $k\in[r-1]$. We observe that since $\b$ is product of elements of $T_{ab}$, then $\b(k)\in\{a,a+1,\ldots,b\},$ for any $k\in\{a,a+1,\ldots,b\}$ and $\b(i)=i$, for any $i\in[n]\setminus\{a,a+1,\ldots,b\}$. 
    
    We distinguish 3 cases:
    \begin{enumerate}
        \item if $\{i_r,j_r\}\cap\{a,a+1,\ldots,b\}=\emptyset$, then  
        inequality~\eqref{eq:lemma_Bruhat_path_2} implies
        \begin{align*}
            i_r=\t{a}{b}\t{i_1}{j_1}\c\t{i_{r-1}}{j_{r-1}}(i_r)>\t{a}{b}\t{i_1}{j_1}\c\t{i_{r-1}}{j_{r-1}}(j_r)=j_r,
        \end{align*}
        but $i_r<j_r$ from which we get a contradiction;
        \item if $i_r<a$ and $j_r\in\{a,a+1,\ldots,b\}$, then by inequality~\eqref{eq:lemma_Bruhat_path_2}, we obtain $i_r>\b\1(j_r);$
        but $\b\1(j_r)\in\{a,a+1,\ldots,b\}$ so, we have a contradiction since $a>i_r$;
        \item if $i_r\in\{a,a+1,\ldots,b\}$ and $j_r>b$, inequality~\eqref{eq:lemma_Bruhat_path_2} yields $\b\1(i_r)>j_r$ and since $\b\1(i_r)\in\{a,a+1,\ldots,b\}$, again we end up with a contradiction.
    \end{enumerate}
    Therefore, necessarily $\t{i_r}{j_r}\in T_{ab}$, so, by induction, every edge of a Bruhat path from $e$ to $\t{a}{b}$ is labeled by a reflection in $T_{ab}$.
\end{proof}

\begin{theorem}
    Let $\s,\q\in S_n$; then 
    $$T_L(\s\jo\q)=T\cap V_{S_n}(\s,\q).$$
\end{theorem}

\begin{proof}
    By \cref{lem:first_inclusion_S_n}, it suffices to show that any transposition $\t{a}{b}\in V_{S_n}(\s,\q)$ belongs to $T_L(\s\jo\q)$. 
    According to \cref{thm:transitive_T_L}, this can be proven by showing that $\t{a}{b}$ lies in the transitive closure of $T_L(\s)\cup T_L(\q)$. In other words, there exists a chain 
    \begin{equation}
        a=i'_1<i'_2<\c<i'_{k+1}=b
    \end{equation} such that $\t{i'_{l-1}}{i'_l}\in T_L(\s)\cup T_L(\q)$, for every $l\in\{2,3,\ldots,k+1\}.$ 
    
    Let us consider a $(\s,\q)$-Bruhat path from $e$ to $\t{a}{b}$:
    \begin{equation}\label{initial_path}
            e\xrightarrow{\t{i_1}{j_1}}\t{i_1}{j_1}\xrightarrow{\t{i_2}{j_2}}\t{i_2}{j_2}\t{i_1}{j_1}\xrightarrow{\t{i_3}{j_3}}\c\xrightarrow{\t{i_{h}}{j_h}}\t{a}{b},
    \end{equation}
    from which 
    \begin{equation}\label{eq:path_product}
        \t{a}{b}=\t{i_{h}}{j_h}\t{i_{h-1}}{j_{h-1}}\c \t{i_1}{j_1}.
    \end{equation}
    
   We define now a permutation $\gamma$ as the product of the transpositions in \eqref{eq:path_product} that are involved in the mapping of $a$ to $b$:
    \begin{equation}\label{eq:path_product_gamma}
     \gamma=\t{i'_{k}}{j'_k}\t{i'_{k-1}}{j'_{k-1}}\c \t{i'_1}{j'_1}.
    \end{equation}
More precisely, 
\begin{itemize}
    \item the transposition $\t{i'_{1}}{j'_1}=\t{a}{j'_1}$ is the first from the right in \eqref{eq:path_product} containing $a$;
    \item for each $l\in\{2,\ldots, k-1\}$, $\t{i'_{l}}{j'_l}$ is the first transposition on the left of $\t{i'_{l-1}}{j'_{l-1}}$ in \eqref{eq:path_product} containing the entry $\t{i'_{l-1}}{j'_{l-1}}\c \t{i'_1}{j'_1}(a)$;
    \item the final transposition is $\t{i'_{k}}{j'_k}=\t{i'_{k}}{b}$.
\end{itemize}  
Clearly, $\gamma(a)=b$. \cref{exp:last_ex} shows a few specific cases.
\smallskip

 Recall that by \cref{lem:T_ab}, all ${i'_{l}}$, ${j'_l}  \in \{a,a+1,\ldots,b\}$ and observe that to show that $\t{a}{b}\in (T_L(\s)\cup T_L(\q))^{tc}$, it is sufficient to show that for any $l\in\{2,\ldots, k\}$
 \begin{equation}\label{passobase}
 {j'_{l-1}}=i'_{l}.
 \end{equation}
    Indeed, under this assumption, we obtain a chain 
    $$a=i'_1<i'_2<\ldots <i'_k<b,$$
    where the transpositions $ \t{a}{i'_{2}}, \t{i'_{2}}{i'_{3}}, \ldots, \t{i'_{k}}{b} \in T_L(\s)\cup T_L(\q)$.\\ 
    We prove \eqref{passobase} by contradiction. So suppose that $p$ is the smallest index in $\{2,\ldots,k\}$ for which this equality does not hold, that is $j'_{p-1}=j'_p$. Recall that by definition, $\t{i'_{p}}{j'_p}$ is the first transposition on the left of  $\t{i'_{p-1}}{j'_{p-1}}$ containing the entry $j'_{p-1}=\t{i'_{p-1}}{j'_{p-1}}\c \t{i'_1}{j'_1}(a)$. 
   \smallskip
    
    Now, consider the vertex $\b$ along the path \eqref{initial_path} from which starts the directed edge labeled by $\t{i'_{p}}{j'_p}$. 
    By definition of $\gamma$, we know that $\b(a)=j'_{p-1}$; moreover, since $i'_{p}<j'_p$ and $\inv(\t{i'_{p}}{j'_p}\b)>\inv(\b),$ we must have $\b\1(j'_p)>\b\1(i'_p).$
    Thus, we obtain
    \begin{equation}\label{eq:lasteq}
        a=\b\1(j'_{p-1})=\b\1(j'_p)>\b\1(i'_p).
    \end{equation}
    However, by \cref{lem:T_ab}, the vertex $\b$ is a product of transpositions from $T_{ab}$, therefore $\b\1(i'_p) \in \{a,a+1,\ldots,b\}$ and this contradicts inequality \eqref{eq:lasteq}. Hence $j'_{p-1}=i'_p$, and the proof is complete.
\end{proof}

We conclude this section with an example aimed at clarifying the content of the previous proof.

\begin{example}\label[example]{exp:last_ex}
    In \cref{fig:Bruhat_paths}, consider the following two distinct $(\s,\q)$-Bruhat paths from $e$ to $\t{2}{4}=1432$: 
    \begin{align*}
        &e\xrightarrow{\t{2}{3}}1324\xrightarrow{\t{3}{4}}1423\xrightarrow{\t{2}{3}}1432=\t{2}{4};\\
        &e\xrightarrow{\t{3}{4}}1243\xrightarrow{\t{2}{4}}1423\xrightarrow{\t{2}{3}}1432=\t{2}{4}.
    \end{align*}
    The permutations $\gamma$ associated with each path, as defined in Equation \eqref{eq:path_product_gamma}, are respectively $\t{3}{4}\t{2}{3}$ and $\t{2}{4}$.
\end{example}

\smallskip
\section{Concluding remarks}

We are working on generalizing the proof of \cref{conjecture_2} for all classical Coxeter groups. Our approach is based on a case-by-case analysis that uses the combinatorial description of any group.

As already said, Dyer's conjecture holds also for the Coxeter groups of type $H_3$ and $F_4$, as we verified using the open-source software Sage.
In type $B$, we have made some progress by adapting the combinatorial approach that was successful in type $A$.
However, this strategy does not seem to extend naturally to type $D$, where additional structural complexities arise.

 At the same time, we are exploring the possibility of finding a uniform proof, possibly using properties of the associated root system or the geometry of the associated Coxeter arrangement.

\smallskip

\bibliographystyle{amsplain}
\bibliography{typeA}

\end{document}